\documentclass[12pt]{article}

\usepackage{wasysym}

\usepackage[english]{babel}
\usepackage{graphicx}

\usepackage[colorlinks= True, citecolor=blue]{hyperref}
\usepackage{amsthm}
\usepackage{amsfonts}
\usepackage{amssymb}

\usepackage{amsmath}
\usepackage{mathrsfs}
\usepackage{mathtools}

\usepackage{xcolor}
\usepackage{xcolor}
\usepackage{ wasysym }

\newtheorem{theorem}{Theorem}[section]
\newtheorem{corollary}[theorem]{Corollary}
\newtheorem{lemma}[theorem]{Lemma}
\newtheorem{proposition}[theorem]{Proposition}

\theoremstyle{definition}

\newtheorem{algorithm}{Exploration}

\newtheorem{remark}[theorem]{Remark}

\usepackage{enumerate}

\usepackage{subcaption}

\newcommand{\mr}{{\textup{\tiny\newmoon}}}
\newcommand{\ml}{{\textup{\tiny\fullmoon}}}

\numberwithin{equation}{section}

\newcommand{\dotms}{\dotsm}

\newcommand{\sV}{\mathscr{V}}

\newcommand{\R}{\mathbb{R}}

\newcommand{\E}{\mathbb{E}}

\newcommand{\cA}{\mathcal{A}}

\newcommand{\weakarrow}{{\overset{(d)}{\Longrightarrow}}}

\newcommand{\PR}{\mathbb{P}}

\newcommand{\bbG}{\mathbb{G}}

\newcommand{\br}{{\operatorname{br}}}
\newcommand{\ex}{{\operatorname{ex}}}

\newcommand{\cS}{\mathcal{S}}

\usepackage[margin = 1in]{geometry}

	\author{David Clancy, Jr.}

 \date{\today}
\title{Asymptotics for the number of bipartite graphs with fixed surplus}

\usepackage{verbatim}
\usepackage{todonotes}

\begin{document}
\maketitle

\begin{abstract}
    In a recent work on the bipartite Erd\H{o}s-R\'{e}nyi graph, Do et al. (2023) established upper bounds on the number of connected labeled bipartite graphs with a fixed surplus. We use some recent encodings of bipartite random graphs in order to provide a probabilistic formula for the number of bipartite graphs with fixed surplus. Using this, we obtain asymptotics as the number of vertices in each class tend to infinity.
\end{abstract}

\section{Introduction}

Cayley's formula gives the number of trees on $n$ labeled vertices as $n^{n-2}$. Equivalently, this counts the number of spanning trees of $K_n$, the complete graph on $n$ vertices. Let $\bbG_n(k)$ denote the collection of connected spanning subgraphs of $K_n$ with exactly $n-1+k$ many edges. In \cite{Wright.77}, Wright established that for each fixed $k$
\begin{align*}
    \#\bbG_{n}(k) \sim \rho_k n^{n-2+ \frac{3k}{2}}
\end{align*}
for some constants $\rho_k>0$. Here, and throughout the article, we write $a_n\sim b_n$ if $a_n/b_n\to1$ as $n\to\infty$. In \cite{Spencer.97}, Spencer gave a probabilistic representation of $\rho_k$ as
\begin{align*}
    \rho_k = \frac{1}{k!} \E\left[\left(\int_0^1 B_{\ex}\,ds\right)^k\right]
\end{align*} where $B_{\operatorname{ex}}$ is a standard Brownian excursion. See \cite{Janson.07} for a more thorough literature review of this connection. See also \cite{Pal.12} for the case of $k = k_n\to\infty$ sufficiently slowly.

Let $K_{n,m}$ be the complete bipartite graph on $n+m$ labeled vertices, where one class has $n$ vertices while the other has $m$ vertices. We let $\bbG_{n,m}(k)$ be the collection of spanning graphs of $K_{n,m}$ with exactly $n+m-1 +k$ many edges. Scoins \cite{Scoins.62} established that
\begin{align}\label{eqn:bpTreeCount}
    \#\bbG_{n,m}(0) = n^{m-1} m^{n-1}.
\end{align}
Recently, in \cite{HSY.22,DEKM.23}, analogues of the result of Wright were established for $k$ fixed and $n,m\to\infty$. Using generating functions, the authors of \cite{HSY.22} show that for each fixed $k$ that
\begin{align*}
    \sum_{n,m: n+m=N} \binom{N}{n} \#\bbG_{n,m}(k) \sim \frac{1}{2^{k-1}} \#\bbG_{N}(k) \qquad \textup{as }N\to\infty.
\end{align*}
In \cite{DEKM.23}, the authors consider local versions and show that as $n,m\to\infty$ with $n/m\in[1/2,2]$
\begin{align*}
    \#\bbG_{n,m}(1) \sim \sqrt{\frac{\pi}{8}} n^{m-\frac12} m^{n-\frac12} \sqrt{n+m}\qquad\textup{and} \qquad \#\bbG_{n,m}(k)\le c_k (n+m)^{3k/2} n^{m-1} m^{n-1}
\end{align*}
for some constants $c_k\to 0$ as $k\to\infty$. The authors of \cite{DEKM.23} can obtain an explicit representation for the asymptotics $\#\bbG_{n,m}(1)$ as the $2$-core of any graph in $\bbG_{n,m}(1)$ is a cycle of even length. 

In this article, we obtain the asymptotics so long as $n/m\to\alpha\in\R_+$ and $k$ is fixed.  More precisely, we establish the following theorem.
\begin{theorem}\label{thm:main}
    Suppose that $n,m\to\infty$ and $n/m\to \alpha\in\R_+$. Then
    \begin{align*}
        \#\bbG_{n,m}(k) \sim (1+\alpha)^{k/2} \rho_k n^{m-1+k/2} m^{n-1 + k} .
    \end{align*}
\end{theorem}

This yields the following corollary.
\begin{corollary}
    Suppose that $\frac{n}{n+m}\to \gamma\in(0,1)$ as $n\to\infty$. Then
    \begin{align*}
        \#\bbG_{n,m}(k) \sim \left(\gamma(1-\gamma)\right)^{k/2} \rho_k  (n+m)^{3k/2}n^{m-1} m^{n-1} .
    \end{align*}
\end{corollary}

\subsection{Overview}

Our proof will be almost entirely probabilistic, in the spirit of Spencer \cite{Spencer.97}.

In Section \ref{sec:Exploration} we describe the (breadth-first) exploration of a graph $G\in \bbG_{n,m}(k)$ for some $k\ge 0$. This exploration gives us two processes $X^\ml, X^\mr$ encoding the number of vertices discovered by each vertex in the exploration. In Section \ref{sec:propofexplo}, we describe the law of these processes when the graph $G$ is a uniformly chosen tree in $\bbG_{n,m}(0)$. In Section \ref{sec:countingGraphs}, we relate $\#\bbG_{n,m}(k)/\#\bbG_{n,m}(0)$ to the expectation of a particular random variable $W_{n,m}$.

In Section \ref{sec:weakconv} we discuss weak convergence. In Section \ref{sec:bridgePoilim}, we prove weak convergence involving some auxiliary processes $Y^{\ml}, Y^\mr$. These processes are connected to $X^\ml, X^\mr$ in Section \ref{sec:vervaatTransform} and to the process $W_{n,m}$ in Section \ref{sec:wnmlim}. In Section \ref{sec:wnmlim} we prove the convergence of the moments of $W_{n,m}$ in order to obtain Theorem \ref{thm:main}.

\section{Exploration of graphs}\label{sec:Exploration}

Let us now explain the exploration of a graph $G\in \bigsqcup_{k\ge 0} \bbG_{n,m}(k)$.  For concreteness, we color the vertices of $K_{n,m}$ as either white or black. We write the vertex set of $K_{n,m}$ as $V_n^\ml\sqcup V_m^\mr$ where $V^\ml_n =\{i^\ml: i\in [n]\}$ are the white vertices and $V^\mr_m = \{i^\mr: i\in[m]\}$ are the black vertices. The exploration is analogous to the explorations in \cite{Federico.19, Wang.23, Clancy.24_Bipartite}. 

We maintain a stack of active vertices that we denote by $\cA_t$ for $t = 0,1,\dotsm, n+m$. We always start with the stack $\cA_{0}$ containing the vertex $v_1:= 1^\ml$. We will define several sequences $(\chi^\ml_j;j\in[n])$, $(\chi^\mr_j;j\in[m])$, and $(\gamma_j;j\in[n+m])$. We will use these sequences to define several processes
\begin{equation}\label{eqn:XNdefs}
    X^\ml(t) = \sum_{s=1}^t \chi_s^\ml,\qquad X^\mr(t) = \sum_{s=1}^t \chi^\mr_s,\qquad N^\ml (t) = \sum_{s=1}^t \gamma_s,\qquad N^\mr(t) = t-N^\ml(t).
\end{equation} 
\begin{algorithm}[Breadth-first exploration of $G$]\label{alg:bfs}
    For $t = 1,2,\dotsm, n+m$ the stack $\cA_{t-1} = (x_1,\dotsm, x_s)$ is of length $s\ge 1$ (by induction). By step $t$, we have explored $a:=N^\ml(t-1)$ many white vertices and $b :=N^\mr(t-1)$ many black vertices. We now explore vertex $x_1$.\\
\indent \textbf{if:} $x_1 = v_{a+1} \in V^\ml_n$, then we find the neighbors of $x_1$ that are either in $\cA_{t}$ or have been unexplored. Each neighbor in $\cA_t$ corresponds to a cycle created. The unexplored ones will be elements of $V^\mr_m$ and we will label these as $w_{b+1},\dotsm, w_{b+r}$ for some $r$ where each $w_{b+j} = i_j^\mr$ for some $i_1<i_2<\dotms<i_r$. Set $\chi_u^\ml = r$ and update
        $
        \cA_{t} = (x_2,\dotsm,x_s, w_{b+1},\dotms, w_{b+r}).
        $ Set $\gamma_{t} = 1$. \\
  \indent  \textbf{else: }$x_1 = w_{b+1}\in V^\mr_m$, then we find the neighbors of $x_1$ that are either in $\cA_{t}$ or have been unexplored. Each neighbor in $\cA_t$ corresponds to a cycle created. The unexplored ones will be elements of $V^\mr_m$ and we will label these as $v_{a+1},\dotsm, v_{a+r'}$ for some $r'
    $ where $w_{a+j} = i_j^\ml$ for some $i_1<i_2<\dotms<i_{r'}$. Set $\chi_u^\mr = r'$ and update
        $
        \cA_{t} = (x_2,\dotsm,x_s, v_{b+1},\dotms, v_{b+r'}).
        $ Set $\gamma_t = 0$.
\end{algorithm}

We call the pair $(X^\ml,X^\mr)$ the child count processes defined via \eqref{eqn:XNdefs} associated with the graph $G$.  Moreover, it is easy to see that $X^\mr\circ X^\ml(t) - X^\mr\circ X^\ml(t-1)$ is the number of white grand-children of the white vertex $v_t$. Hence, standard properties of random trees and \L ukasiewicz paths (see e.g. \cite{LeGall.05}) imply that $Z(t) = X^\mr\circ X^\ml(t)-t$ for $t = 0,1,2,\dotsm, n+m$
has increments in $\{-1,0,1,\dotsm\}$ and satisfies
\begin{equation}\label{eqn:Ztreeencode}
    Z(t) \ge 0 \textup{ for all }t =0,1,\dotms, n-1\qquad \textup{and}\qquad Z(n) = -1.
\end{equation}

Observe that the pair $(\chi^\ml,\chi^\mr)$ constructed in Exploration \ref{alg:bfs} satisfies $\sum_{j=1}^n \chi_j^\ml = m$ and $\sum_{j=1}^m\chi^\mr_j = n-1$ by simply noting which vertices are added to the stack $\cA_t$. Given any two sequences $(\chi^\ml,\chi^\mr)$ of non-negative integers, we can define $X^\ml, X^\mr$ using \eqref{eqn:XNdefs}; however, this need not correspond to a tree $T\in \bbG_{n,m}(0)$. Using the bijection between \L ukasiewicz paths and planar trees (see e.g. \cite{LeGall.05}), we can see that there exists a unique planar tree $T^{\tt plan}$ built from $(\chi^\ml,\chi^\mr)$ whenever the \L ukasiewicz path $Z(t) = X^\mr\circ X^\ml(t)-t$ satisfies \eqref{eqn:Ztreeencode}. We will call such sequences $(\chi^\ml,\chi^\mr)$ \textit{admissible.}

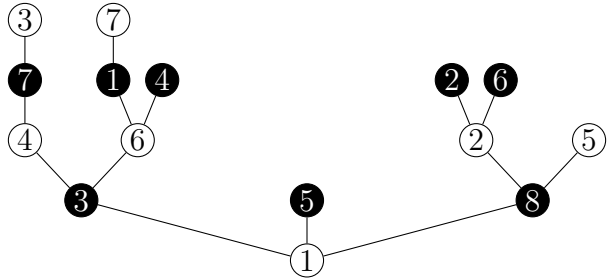
\begin{figure}[h!]
\begin{center}
\begin{tikzpicture}
  [grow'=up, level distance=8mm,
   every node/.style={fill = white,circle, draw , inner sep=1pt},
   level 1/.style={sibling distance=30mm,nodes={fill=black}, text = white},
   level 2/.style={sibling distance=15mm,nodes={fill=white}, text=black} ,
   level 3/.style={sibling distance=6.5mm,nodes={fill=black}, text = white},
   level 4/.style={sibling distance=3mm,nodes={fill=white}, text = black},
   level 5/.style={sibling distance=1mm,nodes={fill=black}, text = white}]
  \node {1}
     child {node{3}
        child {node {4}
            child{ node {7}
                child{node {3}
                }
            }
        }
        child {node {6}
            child{ node {1}
                child{ node {7}
                }
            }
            child{ node{4}
            }
        }
     }
     child {node{5}
     }
     child {node{8}
        child {node {2}
            child{ node{2}
            }
            child{node{6}
            }
        }
        child {node {5}
        }
     }     
     ;
\end{tikzpicture}
\caption{A tree $T$ on $7$ labeled white vertices and $8$ labeled black vertices. When exploring any graph $G$ whose exploration produces the tree $T$ above, the possible edges that need to be checked to find cycles are (listed in order of possible appearance) $5^\mr4^\ml,$ $ 5^\mr6^\ml$, $8^\mr4^\ml$, $8^\mr6^\ml$, $6^\ml7^\mr,$ $ 2^\ml7^\mr$, $2^\ml1^\mr$, $2^\ml4^\mr$, $5^\ml 7^\mr,$ $5^\ml 1^\mr,$ $5^\ml 4^\mr,$ $5^\ml 2^\mr,$ $5^\ml 6^\mr,$ $1^\mr3^\ml$, $4^\mr3^\ml$, $4^\mr7^\ml$,$2^\mr3^\ml$, $2^\mr7^\ml$,$6^\mr3^\ml$, $6^\mr7^\ml$. }
\label{fig:figtree}
\end{center}
\end{figure}
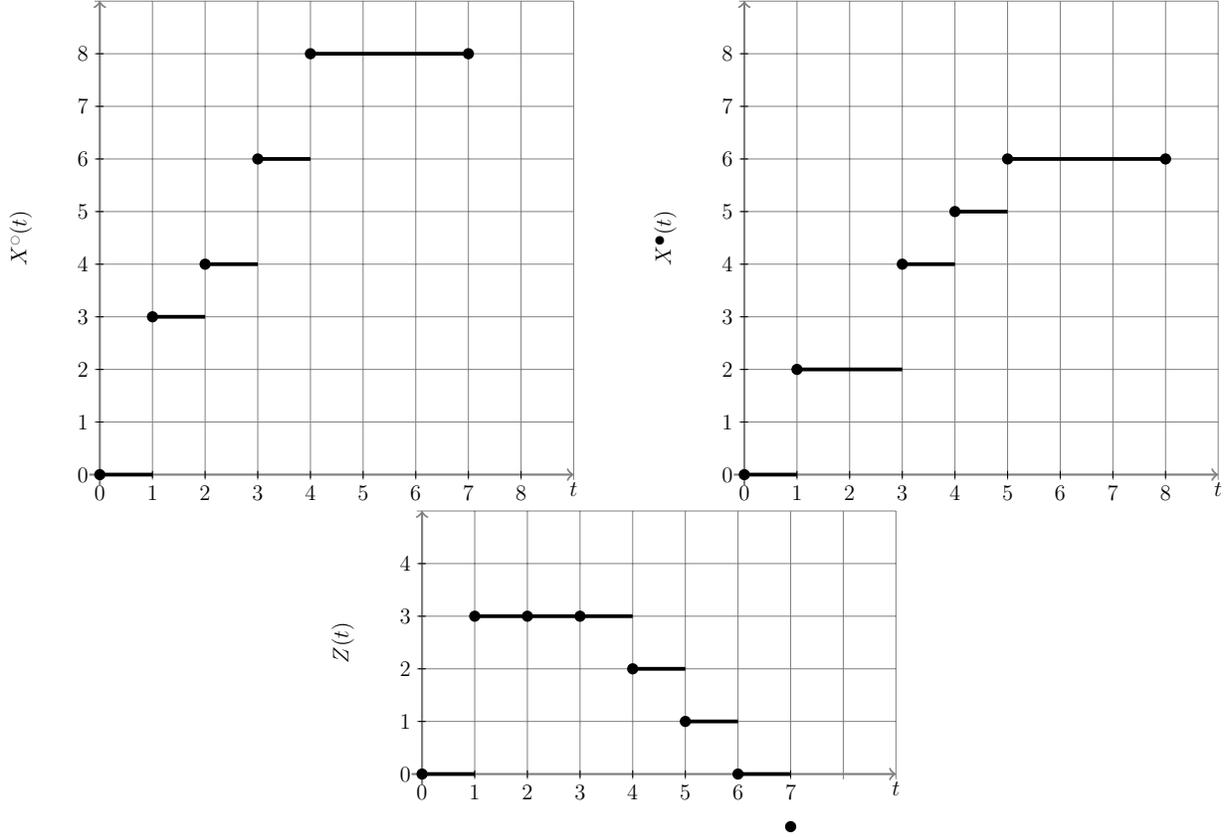
\begin{figure}[t]   
\begin{center}
\begin{tikzpicture}[scale=.7, transform shape]
 \draw[very thin,color=gray] (-0.1,-.1) grid (9,9);
  \draw[thick, ->, color = gray] (-0.2,0) -- (9,0) node[below,text=black] {$t$};
  \draw[thick,->, color = gray] (0,-0.2) -- (0,9);
  \draw (-1.5,4.5) node[text=black, rotate = 90] {$X^\ml(t)$};
    \draw[line width = .5mm] plot[jump mark left, mark=*] 
    coordinates {(-.0,-.0) (1,3) (2,4) (3,6) (4,8) (7,8)};
    \foreach \i [count=\j from 0] in {0, 1, 2, 3, 4, 5, 6, 7, 8}
{
   \draw (\j,2pt) -- ++ (0,-4pt) node[below] {$\j$};
   \draw (2pt,\j) -- ++ (-4pt,0) node[left]  {$\i$};
}
\end{tikzpicture}\qquad 
\begin{tikzpicture}[scale=.7, transform shape]
 \draw[very thin,color=gray] (-0.1,-.1) grid (9,9);
  \draw[thick, ->, color = gray] (-0.2,0) -- (9,0) node[below, text = black] {$t$};
  \draw[thick,->, color = gray] (0,-0.2) -- (0,9);
\draw (-1.5,4.5) node[text=black, rotate = 90] {$X^\mr(t)$};
    \draw[line width = .5mm] plot[jump mark left, mark=*] 
    coordinates {(-.0,-.0) (1,2) (3,4) (4,5) (5,6) (8,6)};
       \foreach \i [count=\j from 0] in {0, 1, 2, 3, 4, 5, 6, 7, 8}
{
   \draw (\j,2pt) -- ++ (0,-4pt) node[below] {$\j$};
   \draw (2pt,\j) -- ++ (-4pt,0) node[left]  {$\i$};
}
\end{tikzpicture}
\begin{tikzpicture}[scale=.7, transform shape]
 \draw[very thin,color=gray] (-0.1,-.1) grid (9,5);
  \draw[thick, ->, color = gray] (-0.2,0) -- (9,0) node[below, text = black] {$t$};
  \draw[thick,->, color = gray] (0,-0.2) -- (0,5);
\draw (-1.5,2.5) node[text=black, rotate = 90] {$Z(t)$};
    \draw[line width = .5mm] plot[jump mark left, mark=*] 
    coordinates {(-.0,-.0) (1,3) (2,3) (3,3) (4,2) (5,1) (6,0) (7,-1)};
       \foreach \i [count=\j from 0] in {0, 1, 2, 3, 4}
{
   \draw (\j,2pt) -- ++ (0,-4pt) node[below] {$\j$};
   \draw (2pt,\j) -- ++ (-4pt,0) node[left]  {$\i$};
}
  \draw (5,2pt) -- ++ (0,-4pt) node[below] {$5$};
   \draw (6,2pt) -- ++ (0,-4pt) node[below] {$6$};
   \draw (7,2pt) -- ++ (0,-4pt) node[below] {$7$};
\end{tikzpicture}
 \caption{The child count processes for the tree $T$ depicted in Figure \ref{fig:figtree}, along with its (white) \L ukasiewicz path.}
    \label{fig:paths}
\end{center}
\end{figure}
\subsection{Properties of the exploration}\label{sec:propofexplo}
The next lemma tells us the law of the child count sequences $(X^\ml,X^\mr)$ for a uniform tree $T\in \bbG_{n,m}(0)$. 
\begin{lemma}
Let $(\chi^\ml,\chi^\mr)$ be admissible. Define $X^\ml(t) = \sum_{s\le t} \chi^\ml_s$ and similarly define $X^\mr$. 
The number of trees $T\in \bbG_{n,m}(0)$ whose child count processes are $X^\ml, X^\mr$ is
\begin{align*}
    \frac{(n-1)!}{\prod_{j=1}^m \chi_j^\mr!} \frac{m!}{\prod_{j=1}^n \chi_j^\ml!} = \binom{n-1}{\chi_1^\mr,\chi_2^\mr,\dotsm, \chi_m^\mr} \binom{m}{\chi_1^\ml,\chi_2^\ml,\dotsm, \chi^\ml_n}.
\end{align*}
\end{lemma}
\begin{proof}
    Let $T^{\tt plan}$ be the rooted planar tree constructed from $(\chi^\ml,\chi^\mr)$. This is uniquely defined by $(\chi^\ml,\chi^\mr)$. We must now assign labels to the vertices in $T^{\tt plan}$ that are consistent with Exploration \ref{alg:bfs} above. The first multinomial coefficient counts the number of ways to assign the $n-1$ labels to the children of the black vertices in increasing order, the second counts the number of ways to assign labels to the children of white vertices.   
\end{proof}

The next lemma gives a probabilistic way to construct $X^\ml$ and $X^\mr$.

\begin{lemma}\label{lem:poissonRep1}
    Let $(\xi^\mr_j;j\ge 1), (\xi^\ml_j;j\ge1)$ be i.i.d. mean 1 Poisson random variables. Define $Y^\ml(t) = \sum_{s\le t} \xi^\ml_s$ and similarly define $Y^\mr$. Let $S(t) = Y^\mr\circ Y^\ml(t)-t$. Let $X^\ml, X^\mr$ be the child count processes for a uniformly chosen random tree $T\in \bbG_{n,m}(0)$. Let $$E_{n,m} = \{ \inf\{t:S(t)=-1\} = n\}\cap \{Y^\ml(n) = m\}.$$ Then $(Y^\ml,Y^\mr)|E_{n,m}\overset{d}{=} (X^\ml,X^\mr).$
   
\end{lemma}
\begin{proof}
Let $(x^\ml,x^\mr)$ be a fixed (deterministic) admissible sequence. The previous lemma gives
\begin{align*}
    \PR\left(X^\ml(t) = \sum_{s\le t}x_s^\ml, X^\mr(t) = \sum_{s\le t} x_s^\mr\right) = \frac{(n-1)! m!}{\#\bbG_{n,m}(0)}\frac{1}{\prod_{s=1}^n x_s^\ml! \prod_{u=1}^m x^\mr_u!}.
\end{align*}

Note that given $E_{n,m}$, $Y^\mr(m) = n-1$ a.s. Hence,if $(x^\mr,x^\ml)$ is admissible we have
\begin{align*}
    \PR&\left(Y^\ml(t) = \sum_{s\le t}x_s^\ml, Y^\mr(t) = \sum_{s\le t} x_s^\mr\bigg|E_{n,m}\right) = \frac{1}{\PR(E_{n,m})} \PR(Y^\ml(t) = \sum_{s\le t}x_s^\ml, Y^\mr(t) = \sum_{s\le t} x_s^\mr,E_{n,m})\\
    &=\frac{1}{\PR(E_{n,m})} \PR(\xi^\ml_s = x^\ml_s, \xi^\mr_u = x^\mr_u\textup{ for all }s,u)=  \frac{1}{\PR(E_{n,m})} \prod_{s=1}^n  \frac{e^{-1}}{x_s^\ml} \prod_{u=1}^m \frac{e^{-1}}{x_u^\mr}.
\end{align*} Both are inversely proportional to ${\prod_{s=1}^n x_s^\ml! \prod_{u=1}^m x^\mr_u!}$, proving the desired statement.
\end{proof}

Since $\#\bbG_{n,m}(0) = n^{m-1}m^{n-1}$ we can see that $\PR(E_{n,m}) = \frac{e^{-(n+m)} n^{m-1}m^{n-1}}{(n-1)!m!}$. We now give a probabilistic proof of this, and hence a probabilistic proof of \eqref{eqn:bpTreeCount}.
\begin{lemma}
    Let $E_{n,m}$ be defined in Lemma \ref{lem:poissonRep1}. Then
    \begin{equation*}
        \PR(E_{n,m}) = \frac{e^{-(n+m)} n^{m-1} m^{n-1}}{(n-1)!m!}.
    \end{equation*}
\end{lemma}
\begin{proof}
    As already noted, we have
    \begin{align*}
        E_{n,m} = \{Y^\ml(n) = m\}\cap \{Y^\mr(m) = n-1\} \cap\{\inf\{t: S(t) = -1\} = n\}.
    \end{align*}
    Set $A_{n,m} = \{Y^\ml(n) = m\}\cap \{Y^\mr(m) = n-1\}$. Since $Y^\ml(n)\sim \operatorname{Poi}(n)$ and $Y^\mr(m)\sim \operatorname{Poi}(m)$ are independent Poisson random variables, $\PR(A_{n,m}) =  \frac{e^{-n}n^m}{m!} \frac{e^{-m}m^{n-1}}{(n-1)!}$. Hence
    \begin{align*}
        \PR&(E_{n,m})=  \PR(\inf\{t: S(t) = -1\} = n|A_{n,m}) e^{-(n+m)} \frac{n^m m^{n-1}}{m!(n-1)!}.
    \end{align*}
    Under $\PR(-|A_{n,m})$ the increments $W_j = S(j)- S(j-1)$ for $j = 1,\dotms, n$ are cyclically exchangeable and $\sum_{j=1}^n W_j = -1$. Hence, by the cyclic lemma (see, e.g. \cite[Lemma 6.1]{Pitman.06}) $\PR(\inf\{t: S(t) = -1\} = n|A_{n,m}) = n^{-1}$. This gives the desired result.
\end{proof}

\begin{remark}
    We discuss this cyclic lemma in more detail in Section \ref{sec:vervaatTransform}.
\end{remark}

\subsection{Counting graphs}\label{sec:countingGraphs}

Let us now turn to the graph counting. In the tree in Figure \ref{fig:figtree}, we see that there are 20 possible edges to $T$ to form a graph $G$ with the same breadth-first spanning tree. It turns out we can represent this as a functional of the child count processes $X^\ml, X^\mr$. As this only depends on the labeled tree $T$, we set
\begin{equation}\label{eqn:Wdef}
W = W(T) = - m(n-1) + \sum_{s=0}^{n-1} X^\ml(s) + \sum_{u=0}^{m-1} X^\mr(u) 
\end{equation}
where $X^\ml, X^\mr$ is the child count process of $T$. It is easy to see that using Figure \ref{fig:paths} that for the tree $T$ in Figure \ref{fig:figtree} that
$$\sum_{s=0}^{7-1} X^\ml(s) = 37 \qquad \textup{and}\qquad \sum_{u=0}^{8-1} X^\mr(u) = 31,$$ and, therefore, $W(T) = 37+31-48 = 20$.
\begin{proposition}\label{prop:Radon-Nikodym}
    Suppose that $T\in \bbG_{n,m}(0)$ is fixed. Then the number of graphs $G\in \bbG_{n,m}(k)$ whose spanning tree is $T$ is $\binom{W(T)}{k}$. In particular,
\begin{align*}
    \#\bbG_{n,m}(k) = \E\left[\binom{W(T)}{k}\right] n^{m-1}m^{n-1}\qquad\textup{where }\qquad T\sim \operatorname{Unif}(\bbG_{n,m}(0)).
\end{align*}
\end{proposition}
\begin{proof}
    Observe that in the exploration, a surplus edge can be added exclusively when we are exploring a white vertex (resp. black vertex) at time $t$ and pair it with a black vertex (resp. white vertex) in the stack $\cA_{t-1}$.

    Let us now look at the stack $\cA_{t-1}$. The top of the stack is white if $\gamma_t = 1$ and black if $\gamma_t = 0$. The vertices that have been discovered up-to and including step $t-1$ are those discovered by vertices $(v_s; s\le1+ N^\ml(t-1))$ and $(w_u: u\le N^\mr(t-1))$. Moreover, those vertices have all been removed from the stack by time $t$. Hence, $\cA_{t-1}$ consists of $X^\ml(N^\ml(t-1)) - N^\mr(t-1)$ many black vertices and $1+X^\mr(N^\mr(t-1)) - N^\ml(t-1)$ many white vertices.   Therefore, the total number of possible cycles that can be added is 
    \begin{align*}
        W' &= \sum_{t=1}^{n+m}\left[ (X^\ml\circ N^\ml(t-1) - N^\mr(t-1))\gamma_t +  (X^\mr\circ N^\mr(t-1)-N^\ml(t-1)) (1-\gamma_t)\right].
    \end{align*} Hence there are $\binom{W'}{k}$ many graphs $G\in \bbG_{n,m}(k)$ whose spanning tree is $T$. We claim that $W' = W(T)$. To see this, note that for each $s = 0,1,\dotsm, n-1$ there is precisely one $t=1,\dotms, n+m$ such that $s 
 = N^\ml(t-1)$ and $\gamma_t = 1$. Hence, for any $f$
 \begin{align*}
     \sum_{t=1}^{n+m} f(N^\ml(t-1))\gamma_t = \sum_{s=0}^{n-1} f(s).
 \end{align*}
 Similarly, $\sum_{t=1}^{n+m} f(N^\mr(t-1)) (1-\gamma_t) = \sum_{s=0}^{m-1} f(s)$. Since $N^\ml(t) = t-N^\mr(t)$ and $\gamma_t = N^\ml(t)-N^\ml(t-1)$ we get
 \begin{align*}
     \sum_{t=1}^n& (X^\ml\circ N^\ml(t-1) - N^\mr(t-1)) \,\gamma_t = \sum_{s=0}^{n-1} (X^\ml(s) +s) - \sum_{t=1}^{n+m} (t-1) \gamma_t.
 \end{align*} A similar formula can be established for the other sums against $(1-\gamma_t)$:
 \begin{align*}
\sum_{t=1}^n& (1+X^\mr\circ N^\mr(t-1) - N^\ml(t-1)) \,(1-\gamma_t) = \sum_{s=0}^{m-1} (1+X^\mr(s) +s) - \sum_{t=1}^{n+m} (t-1) (1-\gamma_t).
 \end{align*}Hence
 \begin{align*}
   W' &= \sum_{s=0}^{n-1}X^\ml(s) + \sum_{s=0}^{m-1} X^\mr(s) + \binom{n}{2}+\binom{m}{2} + m - \sum_{t=1}^{n+m} (t-1)(\gamma_t+1-\gamma_t)\\
   &= \sum_{s=0}^{n-1}X^\ml(s) + \sum_{s=0}^{m-1} X^\mr(s) + \binom{n}{2}+\binom{m}{2} + m - \binom{n+m}{2} = W(T).
 \end{align*}
 \end{proof}

 \section{Scaling Limits}\label{sec:weakconv}

 We are now left to investigate the random variable $W_{n,m}:= W(T)$ where $T\sim \operatorname{Unif}(\bbG_{n,m}(0))$. We will henceforth include the subscript $n$ in all the processes and assume that $m = m_n$ also depends on $n$. To do this, it is easier to start with the scaling limits for $Y^\ml_n$ and $Y^\mr_n$ from Lemma \ref{lem:poissonRep1}. We will also let $A_{n,m}$ be the event
 \begin{align}\label{eqn:AnmDef}
     A_{n,m} = \{Y^\ml_n(n) = m, Y^\mr_n(m) = n-1\}.
 \end{align}

\subsection{Fluctuations of Poisson Bridges}\label{sec:bridgePoilim}

We start with the following consequence of Donsker's theorem for empirical processes. In the sequel, given any function $f:[n]\to \R$ we extend this to all of $[0,n]\subset\R$ by setting $f(t) = f(\lfloor t\rfloor)$.

 \begin{proposition}\label{prop:PoissonBridgeLimit}
    Let $Y^\ml = Y^\ml_n$ and $Y^\mr = Y^\mr_n$ be as in Lemma \ref{lem:poissonRep1} for some sequence $m = m_n\to\infty$ as $n\to\infty$. Conditionally given $A_{n,m}$ from \eqref{eqn:AnmDef}, it holds jointly in the $J_1$ topology
    \begin{align}
      \label{eqn:Yl} \left(\sqrt{m} \left(m^{-1} Y^\ml(nt) - t\right);{t\in[0,1]} \right)\big| A_{n,m}\weakarrow B^\ml_{\operatorname{br}}\\
      \label{eqn:Yr} \left( \sqrt{n} \left(n^{-1} Y^\mr(mt) - t\right);{t\in[0,1]}\right) \big| A_{n,m}\weakarrow B^\mr_{\operatorname{br}}
    \end{align}
    for two independent Brownian bridges $B^\ast_{\operatorname{br}}$, $\ast\in \{\ml,\mr\}$. 
\end{proposition}
\begin{proof}
    Let $U_j$ be independent and uniformly distributed on $[0,1]$. For each $n\ge 1$, define $F_n(t) =  n^{-1}\sum_{j=1}^n 1_{[U_j\le t]}.$ 
    By Donsker's theorem \cite{Marckert.08}
    \begin{align*}
        \left(\sqrt{n}(F_n(t)-t);t\in[0,1]\right) \weakarrow B_{\operatorname{br}}
    \end{align*} for a Brownian bridge $B_{\operatorname{br}}$.
    Also, by standard properties about Poisson processes and uniform random variables, it is easy to see that
    \begin{equation}\label{eqn:yandunif}
        \left(m^{-1}Y^\ml(t); t = 0,1,\dotsm,n \right)\big|A_{n,m} \overset{d}{=} \left(F_m(t/n); t = 0,1,\dotsm,n\right)
    \end{equation} 
    An application of \cite[pg 146]{Billingsley.99} that gives the convergence \eqref{eqn:Yl}. The limit in \eqref{eqn:Yr} is similar.
\end{proof}

In the sequel, we will need some more precise control on the growth of the processes $m^{-1} Y_{n}^\ml(nt)-t$ and $n^{-1} Y_n^\mr(mt) - t$. As in the proof of Proposition \ref{prop:PoissonBridgeLimit}, we can do this by a uniform empirical process. The following lemma follows from the proof Lemma 13 in \cite{ABBG.12}.
\begin{lemma}
    Let $(U_j;j\ge1)$ be i.i.d. uniform $[0,1]$. Then there is a universal constant $C, \lambda>0$ such that for all $n\ge 1$
    \begin{align*}
        \PR\left(\sup_{t\in[0,1]} |\sqrt{n}(F_n(t)-t)| \ge x\right) \le C \exp(-\lambda x^2).
    \end{align*}
\end{lemma}
Looking at \eqref{eqn:yandunif} the previous lemma gives the following.
\begin{corollary}\label{cor:uicor}
    Suppose the assumptions of Proposition \ref{prop:PoissonBridgeLimit}. There exists constants $C,\lambda>0$ such that
    \begin{align*}
        \PR\left(\sup_{t\in[0,1]} \left|\sqrt{m}(m^{-1} Y_n^\ml(nt)-t) \right| + \sup_{t\in[0,1]} \left|\sqrt{n} (n^{-1}Y_n^\mr(mt)-t)\right|> x\bigg|A_{n,m}\right) \le C e^{-\lambda x^2}.
    \end{align*}
\end{corollary}

\subsection{Fluctuations of $S_n$}
The convergence of the process $S_n(t) = Y_n^\mr\circ Y^\ml_n(t) - t$ given $A_{n,m}$ follows from Proposition \ref{prop:PoissonBridgeLimit} and a standard result on the fluctuations of compositions (see \cite{Vervaat.72, Whitt.80} or \cite[Section 13.3 ]{Whitt.02}). We recall this with the next lemma.
\begin{lemma}\label{lem:vervaat1}
    Suppose that $x_n,y_n:[0,1]\to[0,1]$ are non-decreasing c\`adl\`ag functions, and $\psi, \varphi:[0,1]\to \R$ are continuous functions. Let $c_n\to\infty$ and suppose that in the Skorohod $J_1$ topology both $(c_n(x_n(t)-t);t\in[0,1]) \to \psi$ and $(c_n(y_n(t)-t);t\in[0,1])\to \varphi.$ Then
    \begin{align*}
        (c_n(x_n\circ y_n(t)-t);t\in[0,1]) \to \psi+\varphi.
    \end{align*}
\end{lemma}

The preceding lemma and Proposition \ref{prop:PoissonBridgeLimit} imply the following.
\begin{corollary}\label{cor:sconverge}
     Let $Y^\ml = Y^\ml_n$, $Y^\mr = Y^\mr_n$ and $S = S_n$ be as in Lemma \ref{lem:poissonRep1} for some sequence $m = m_n\to\infty$ as $n\to\infty$. Let $A_{n,m} = \{Y^\ml(n) = m\}\cap \{Y^\mr(m) = n-1\}$.
     Suppose that $n/m_n\to \alpha\in\R_+$. Then, jointly with the convergence in Proposition \ref{prop:PoissonBridgeLimit}, 
     \begin{align*}
         (n^{-1/2} S_n(nt);t\in[0,1])|A_{n,m} \weakarrow \left(\sqrt{\alpha}B_{\operatorname{br}}^\ml + B_{\operatorname{br}}^\mr \right).
     \end{align*}
\end{corollary}
\begin{proof}
    Let $\overline{Y}_n^\ml(t) = m^{-1} Y_n^\ml(nt)$ and $\overline{Y}_n^\mr = n^{-1}Y_n^\mr(mt)$. Then $n^{-1} S_n(nt) =  \overline{Y}_n^\mr \circ\overline{Y}_n^\ml(t)  - t.$
    By Proposition \ref{prop:PoissonBridgeLimit} we have
    \begin{align*}
        \left(\sqrt{n} \left( \overline{Y}_n^\mr (t)-t\right);t\in[0,1]\right)|A_{n,m} \weakarrow B^\mr_{\operatorname{br}}.
    \end{align*} Since $n/m_n\to \alpha$, we see that $\sqrt{n}\sim \sqrt\alpha \sqrt{m}$ and so
    \begin{align*}
        \left(\sqrt{n} (\overline{Y}^\ml_n(t) -t );t\in[0,1]\right) |A_{n,m}\weakarrow \sqrt{\alpha} B_{\operatorname{br}}^\ml.
    \end{align*} The result now follows by Lemma \ref{lem:vervaat1}.
\end{proof}

\subsection{The Vervaat Transform}\label{sec:vervaatTransform}

In this section we discuss in more detail the connection between $S_n |A_{n,m}$ and $S_n |E_{n,m}$ where $E_{n,m}$ is defined in Lemma \ref{lem:poissonRep1} as well as the connection to the scaling limits. 

We start with the operation in the continuum. Let $f:[0,1]\to \R$ be a c\`adl\`ag process without any negative jumps such that $f(0) = f(1) = f(1-) = 0$. Note that $f$ need not attain its global minimum; however, $t\mapsto f(t-)$ will. We let $\tau$ be the first time that $f(t-)$ attains the global infimum of $f$. That is
\begin{align*}
    \tau = \inf\{t: f(t-) = \inf_{s\in[0,1]} f(s)\}
\end{align*}
We extend $f:[0,1]\to\R$ to $f:\R\to\R$ by setting $f(t) = f(\{t\})$ where $\{t\}$ is the fractional part of $t$. The \textit{Vervaat transform of }$f$ is
\begin{align*}
    \sV(f)(t) = f( \tau + t) - f(\tau).
\end{align*}
In words, the Vervaat transform exchanges the pre- and post-infimum parts of $f$. See \cite{AUB.20,Bertoin.01,Vervaat.79} for more details. In \cite{Vervaat.79}, Vervaat proved that
\begin{align}\label{eqn:vervaat}
    (\sV(B_\br)(t);t\in[0,1]) \overset{d}{=} (B_{\operatorname{ex}}(t);t\in[0,1])
\end{align} where $B_{\operatorname{ex}}$ is a standard Brownian excursion.

The discrete Vervaat transform is defined slightly differently. First, consider the discrete bridge $f$ of length $n$ from $0$ to $-1$ which is of the form 
\begin{align*}
    f_n(k) = \sum_{j=1}^k (x_j-1),\qquad\textup{for}\qquad k = 0,1,\dotsm, n
\end{align*} where $x_j\in \{0,1,\dotsm\}$ and suppose that $f_n(n) = -1$. We call such a function $f_n$ a \textit{downward skip-free bridge} of length $n$. Similar to above, define
\begin{align*}
    \tau_n = \min\{k: f_n(k) = \min_{0\le j\le n} f_n(j)\}.
\end{align*}
We define the discrete Vervaat transform $\sV_n$ as
\begin{align*}
    \sV_n(f_n)(k) = \sum_{j=1}^k (x_{j+\tau_n} - 1)
\end{align*}where the index $j+\tau_n$ is interpreted modulo $n$. 

The following lemma is elementary. See \cite[Lemma 3]{Bertoin.01} or \cite[Lemma 14]{Kersting.11}.
\begin{lemma}\label{lem:vervaat2}
    Suppose that $(f_n;n\ge 1)$ is a sequence of downward skip-free bridges of length $n$ and that $f$ is a c\`adl\`ag bridge with no negative jumps. Suppose that $0<\delta_n\to0$ is a sequence of constants such that in the $J_1$ topology
    \begin{align*}
        \left(\delta_n f_n (nt) ;t\in[0,1]\right) \to (f(t);t\in[0,1]).
    \end{align*}
    Suppose that $f$ attains its global minimum uniquely and continuously, i.e. $f(\tau-) = f(\tau)<f(t)$ for all $t\neq \tau$. Then $\tau_n/n\to \tau$ and
    \begin{align*}
        \left(\delta_n\sV_n(f_n)(nt);t\in[0,1]\right)\longrightarrow (\sV(f)(t);t\in[0,1]).
    \end{align*}
\end{lemma}

An important property of the discrete Vervaat transform is how it interacts with exchangeable increment processes. For this, it will be better to define cyclic shifts more generally. Let $g_n(k) = \sum_{j=1}^k y_j$ where $y_j\in \R$ for $j\in[n]$. We define
\begin{align*}
    \theta_n(g_n,i)(k) = \sum_{j=1}^k y_{j+i}
\end{align*} where, again, we interpret the index $j+i$ as its equivalence class modulo $n$. Note
\begin{equation}
   \label{eqn:compRuleForTheta} \theta_n\left( \theta_n(g_n,i_1), i_2\right) = \theta_n (g_n,i_1+i_2)
\end{equation}and $\theta_n(f_n,\tau_n) = \sV_n(f_n)$. Also, for $i\in\{0,1,\dotsm,n-1\}$ we have 
\begin{align}\label{eqn:theta_nforpath}
    \theta_n(g_n,i)(k) = \begin{cases}
        g_n(k+i)-g_n(i) &: i+k\le n\\
        g_n(k+i-n)-g_n(i) + g_n(n) &: n<i+k\le i+n.
    \end{cases}
\end{align}
The following lemma easily follows from the above observation.
\begin{lemma}\label{lem:theta_nLinear}
    For each $i$, $\theta_n(-,i)$ is linear on the collection of functions $g_n(k) = \sum_{j=1}^k y_j$. Moreover, if $g_n(k) = ck$ for some constant $c\in \R$, then $\theta_n(g_n,i) = g_n$ for all $i$. In particular,
    \begin{align*}
        \theta_n(g_n - c\operatorname{Id},i) = \theta_n(g_n,i) - c\operatorname{Id}
    \end{align*}
    where $\operatorname{Id}(k) = k$ for all $k$.
\end{lemma}

We include the following lemma containing the main results in Section 6.1 of \cite{Pitman.06}.
\begin{lemma}\label{lem:discVervaat}
    Suppose that $\xi_j$ are i.i.d. random variables such that $\PR(\xi_j\ge 0) = 1$ and $\PR(\sum_{j=1}^n (\xi_j -1)= -1)>0$. Let $S(k) = \sum_{j=1}^k (\xi_j-1)$. Let $$A_n = \{S(n) = -1\}\quad\textup{and}\quad E_n = \{S(n) = -1, S(k) \ge 0\textup{ for }k\le n-1\}.$$ Then the following hold
    \begin{enumerate}
        \item $\PR(E_n) = n^{-1} \PR(A_n)$.
        \item Let $\tau_n = \inf\{k: S(k) = \min_{k\le n}S(k)\}$. Then $\tau_n|A_n \sim \operatorname{Unif}\{1,2,\dotsm,n\}$.
        \item $\sV_n(S)| A_n \overset{d}{=} S|E_n$.
        \item Let $U_n \sim \operatorname{Unif}\{1,2,\dotms,n\}$ be independent of $S$. Then
        \begin{align*}
            (\tau_n, \theta_n(S,\tau_n))|A_n \overset{d}{=} (U_n, S)|E_n.
        \end{align*}
    \end{enumerate}
\end{lemma}

Let us fix an $n,m$ and let $X^\ml_n, X^\mr_n$ be the child count processes associated with a uniformly chosen tree $T\in \bbG_{n,m}(0)$. Recalling Lemma \ref{lem:poissonRep1}, we see that
\begin{align*}
    (X^\ml_n,X^\mr_n, Z_n) \overset{d}{=} (Y^\ml_n,Y^\mr_n, S_n)|E_{n,m}.
\end{align*} Recall that $Y_n^\mr = (Y_n^\mr(k);k=0,1,\dotms,m)$, $Y^\ml_n = (Y^\ml_n (k); k = 0,1,\dotms,n)$ and
\begin{align*}
    S_n(k) = Y^\mr_n(Y^\mr_n(k)) - k.
\end{align*} In particular, the (unconditioned) increments of $S_n$ are i.i.d. From here, it is not hard to see using Lemma \ref{lem:discVervaat}(4) that if $U_n\sim \operatorname{Unif}\{1,\dotms,n\}$ is independent of $Y_n^\ml, Y_n^\mr$ then 
\begin{align*}
   \Big(U_n,Y_n^\ml, Y_n^\mr, S_n \Big)| E_{n,m} \overset{d}{=} \Big(\tau_n,\theta_n(Y_n^\ml,\tau_n), \theta_m(Y_n^\mr,Y_n^\ml(\tau_n)), \theta_n( S_n,\tau_n)\Big)|A_{n,m}
\end{align*}
where $\tau_n = \inf\{k: S_n(k) = \min_{j\le n}S_n(j)\}$.
Consequently, if $U_n$ is independent of $X^\ml_n,X^\mr_n$ then
\begin{align*}
   \Big(U_n,X^\ml_n, X^\mr_n) \overset{d}{=}\Big(\tau_n,\theta_n(Y_n^\ml,\tau_n), \theta_m(Y_n^\mr,Y_n^\ml(\tau_n))\Big)|A_{n,m}.
\end{align*} 
An application of Lemma \ref{lem:theta_nLinear} gives the following lemma.
\begin{lemma}\label{lem:vervaat4}
Maintain the notation above and let $\operatorname{Id}(k) = k$ for all $k$. Then
\begin{align*}
    \left((X_n^\ml - \frac{m}{n} \operatorname{Id}),(X_n^\mr - \frac{m}{n} \operatorname{Id}) \right) \overset{d}{=} \Big(\theta_n(Y_n^\ml,\tau_n)- \frac{m}{n}\operatorname{Id}, \theta_m(Y_n^\mr,Y_n^\ml(\tau_n))-\frac{n}{m}\operatorname{Id}\Big)|A_{n,m}.
\end{align*}
\end{lemma}

\subsection{Scaling limit of $W_{n,m}$}\label{sec:wnmlim}

Recall the definition of $W_{n,m}$ in \eqref{eqn:Wdef} for a uniform tree $T_n\in \bbG_{n,m}(0)$ for some sequence $m_n\to\infty$. The purpose of this section is to prove the following proposition. To state it, we let $\cS(t) = \sqrt{\alpha} B_{\br}^\ml(t) + B_\br^\mr(t) = \sqrt{1+\alpha}B_\br(t)$ where $B_\br$ is a standard Brownian bridge. Let $\tau = \inf\{u: \cS(u) = \inf_{t\in[0,1]} \cS(t)\}.$ By elementary properties of Brownian bridges, almost surely $\tau$ is the unique global minimum of $\cS$ on $[0,1]$. By \eqref{eqn:vervaat}
\begin{equation*}
(\sV(\cS)(t);t\in[0,1]) \overset{d}{= }(\sqrt{1+\alpha} \,B_{\operatorname{ex}}(t);t\in[0,1]).
\end{equation*}

\begin{proposition}\label{prop:Wlim}
    Let $T\sim \operatorname{Unif}(\bbG_{n,m}(0))$ where $n,m\to\infty$ and $n/m\to\alpha\in \R_+$. Then
    \begin{align*}
        \frac{1}{m\sqrt{n}}W_{n,m} \weakarrow  \sqrt{1+\alpha}\int_0^1 B_{\operatorname{ex}}(t)\,dt.
    \end{align*}
    Moreover, for any $k$ fixed,
    \begin{align*}
       \frac{1}{m^k n^{k/2} }\E\left[W_{n,m}^k\right] \to (1+\alpha)^{k/2} \E\left[\left(\int_0^1 B_{\operatorname{ex}}(t)\,dt\right)^k\right].
    \end{align*}
\end{proposition}

We begin with some algebraic manipulations. Note from the definition of $W_{n,m}$ in \eqref{eqn:Wdef}
\begin{align*}
    W_{n,m} 
    &= \sum_{j=0}^{n-1} \left(X_n^\ml(j)-  \frac{m}{n} j\right) + \sum_{j=0}^{m-1} \left(X_n^\mr(j)- \frac{n}{m}j\right) + \frac{m}{n}\binom{n}{2} + \frac{n}{m}\binom{m}{2}  - m(n-1)\\
    &= \sum_{j=0}^{n-1} \left(X_n^\ml(j)-  \frac{m}{n} j\right) + \sum_{j=0}^{m-1} \left(X_n^\mr(j)- \frac{n}{m}j\right) + \frac{m-n}{2}.
\end{align*}
Since $X_n^\ml(n) = m$ and $X^\mr_n(m) = n-1$, we have 
\begin{equation}\label{eqn:Wequivlaw1}
    W_{n,m} = \sum_{j=0}^{n} \left(X_n^\ml(j)-  \frac{m}{n} j\right) + \sum_{j=0}^{m} \left(X_n^\mr(j)- \frac{n}{m} j\right) + \frac{m-n+2}{2}.
\end{equation}

 We now relate these two summations above to the processes $Y_n^\ml, Y_n^\mr$ given $A_{n,m}$. To do this, we will use Lemma \ref{lem:vervaat4} and the following lemma, which is a direct consequence of \eqref{eqn:theta_nforpath}.
\begin{lemma}\label{lem:thetaAndInt}
    Suppose that $g_n(k) = \sum_{j=1}^k y_j$ for $k=0,1,\dotsm,n $ for some $y_j\in \R$. Then for $i\in\{0,1,\dotsm,n-1\}$
    \begin{align*}
        \sum_{j=0}^n \theta_n(g_n,i)(j) = ig_n(n) - ng(i)+ \sum_{j=0}^n g_n(j) .
    \end{align*}
\end{lemma}

We now prove the following lemma.
\begin{lemma}\label{lem:WnewLaw}
    Let $Y_n^\ml, Y_n^\mr, S_n$ be as in Lemma \ref{lem:poissonRep1} and let $A_{n,m}$ be defined as in \eqref{eqn:AnmDef}. Let $\tau_n = \min\{k: S_n(k) = \min_{j\le n} S_n(j)\}$. Then
    \begin{align}
        \nonumber\Bigg(&\sum_{j=0}^n (Y_n^\ml(j)- \frac{m}{n}j)  - n \left( Y_n^\ml(\tau_n)- \frac{m}{n}\tau_n\right) \\
     \label{eqn:Wequivlaw2}   &\qquad +\sum_{j=0}^m (Y_n^\mr(j) - \frac{n}{m}j)  -m \left( Y_n^\mr\circ Y_n^\ml(\tau_n)-\frac{n}{m} Y_n^\ml(\tau_n)\right)- Y_n^\ml(\tau_n)\\
      \nonumber  &\qquad \qquad + \frac{m-n+2}{2} \Bigg)|A_{n,m} \overset{d}{=} W_{n,m}.
    \end{align}
\end{lemma}
\begin{proof}
    This is just a combination of Lemmas \ref{lem:vervaat4} and \ref{lem:thetaAndInt}. Indeed,
    \begin{align*}
        \sum_{j=0}^n &(X_n^\ml(j)-\frac{m}{n}j) \overset{d}{=} \sum_{j=0}^n \left(\theta_n(Y_n^\ml- \frac{m}{n}\operatorname{Id},\tau_n)(j)\right)  \bigg|A_{n,m}.
    \end{align*}
    Moreover, almost surely on the event $A_{n,m}$ we have $Y_n^\ml(n) = m$ and so an application of Lemma \ref{lem:thetaAndInt} gives (a.s. on $A_{n,m}$)
    \begin{align*}
      \sum_{j=0}^n& \left(\theta_n(Y_n^\ml- \frac{m}{n}\operatorname{Id},\tau_n)(j)\right)   = \tau_n \left(Y_n^\ml(n)-m\right) - n \left(Y_n^\ml(\tau_n)-\frac{m}{n}\tau_n\right)
      \sum_{j=0}^n \left(Y_n^\ml(j)-\frac{m}{n}j\right)\\
      &=\sum_{j=0}^n \left(Y_n^\ml(j)-\frac{m}{n}j\right) - n \left(Y_n^\ml(\tau_n)-\frac{m}{n}\tau_n\right).
    \end{align*}
    This gives the first term in \eqref{eqn:Wequivlaw2}. The second term is obtained similarly. The result follows from \eqref{eqn:Wequivlaw1}.
\end{proof}

Now, in order to establish scaling limits for $W_{n,m}$ we just need to establish conditional scaling limits for each of the terms appear in \eqref{eqn:Wequivlaw2}. The next lemma handles this.
\begin{lemma}\label{lem:helps1}
    Maintain the notation as Lemma \ref{lem:WnewLaw}. Then jointly the following convergences hold conditionally given $A_{n,m}$
    \begin{align*}
        &\frac{1}{m\sqrt{n}} \sum_{j=0}^n (Y_n^\ml(j)-\frac{m}{n}j) \weakarrow \int_0^1 \sqrt{\alpha} B_{\br}^\ml(t)\,dt,&&
        \frac{1}{m\sqrt{n}} \sum_{j=0}^n (Y_n^\mr(j)-\frac{n}{m}j) \weakarrow \int_0^1  B_{\br}^\mr(t)\,dt,\\
        &\frac{\sqrt{n}}{m} (Y_n^\ml(\tau_n) - \frac{m}{n}\tau_n) \weakarrow \sqrt{\alpha}B^\ml_\br(\tau), &&\frac{1}{\sqrt{n}} \left(Y_n^\mr\circ Y^{\ml} (\tau_n) - \frac{n}{m}Y_n^\ml(\tau_n) \right)\weakarrow B^\mr_\br(\tau).
    \end{align*}
    Moreover, all the prelimits above are uniformly bounded in $L^p(\Omega, \PR(-|A_{n,m}))$ for all $p\ge 1$. 
\end{lemma}
\begin{proof} The uniform bound in $L^p$ follows Corollary \ref{cor:uicor}. 

    Let us write $\overline{Y}_n^\ml(t) = m^{-1}Y_n^\ml(nt)$ and $\overline{Y}_n^\ml(t) = m^{-1} Y_n^\mr(t)$ as we did in the proof of Corollary \ref{cor:sconverge}. Throughout this proof, we work conditionally given $A_{n,m}$. By applying Proposition \ref{prop:PoissonBridgeLimit} and the continuity of $f\mapsto \int_0^1 f(s)\,ds$ we see 
    \begin{align*}
        \frac{1}{m\sqrt{n}}\sum_{j=0}^n &(Y_n^\ml(j)-\frac{m}{n}j) = \frac{1}{m\sqrt{n}}\int_0^{n+1} \left(Y_{n}^\ml(\lfloor s\rfloor )- \frac{m}{n} \lfloor s\rfloor \right)\,ds\\
        &=\sqrt{n} \int_0^{1+n^{-1}}\left( \overline{Y}_n^\ml(t)-\frac{\lfloor nt\rfloor}{n}\right)\,dt  = \sqrt{{\frac{n}{m}}} \int_0^{1+o(1)} \sqrt{m}\left(\overline{Y}_n^\ml(t)-\frac{\lfloor nt\rfloor}{n}\right)\,dt\\
        &\weakarrow \sqrt{\alpha} \int_0^1 B_\br^\ml(t)\,dt.
    \end{align*}
    Similarly, $\frac{1}{m\sqrt{n}} \sum_{j=1}^m(Y_n^\mr(j)-\frac{n}{m}j) \weakarrow \int_0^1 B_\br^\mr (t)\,dt.$

    Also, by Corollary \ref{cor:sconverge} and Lemma \ref{lem:vervaat2} we see that $\tau_n/n\to \tau = \inf\{t: \cS(t) = \inf_{u\in[0,1]} \cS(u)\}$ jointly with the convergence in Proposition \ref{prop:PoissonBridgeLimit}. By Proposition 2.1 in \cite[Chapter VI]{JS.13}, if $t_n\to t$, $f_n\to f$ in the $J_1$ topology, and $f(t) = f(t-)$ then $f_n(t_n)\to f(t)$. Therefore,
    \begin{align*}
        \frac{\sqrt{n}}{m} \left(Y_n^\ml(\tau_n) - \frac{m}{n}\tau_n\right) = \sqrt{n} \left(\overline{Y}_n^\ml(\tau_n/n) - \frac{\tau_n}{n}\right) \weakarrow \sqrt{\alpha} B_\br^\ml(\tau)
    \end{align*}
    where we used the previous observation with $t_n = \tau_n/n$ and $f_n = \sqrt{n}(\overline{Y}_n^\ml(t)-t)\weakarrow \sqrt{\alpha}B_\br^\mr(t)$. The other term is analogous. Indeed,
    \begin{align*}
        \frac{1}{\sqrt{n}} \left(Y_n^\mr\circ Y_n^\ml(\tau_n)  - \frac{n}{m}Y_n^\ml(\tau_n)\right) = \sqrt{n}\left(\overline{Y}_n^\mr\circ \overline{Y}_n^\ml(\tau_n/n)  - \overline{Y}_n^\ml(\tau_n/n)\right).
    \end{align*}
    By Proposition \ref{prop:PoissonBridgeLimit} $\overline{Y}_n^\ml(t)\weakarrow t$ locally uniformly in $t$ and, in combination with \cite[pg 146]{Billingsley.99}, we have in the $J_1$ topology $$\sqrt{n}(\overline{Y}_n^\mr\circ \overline{Y}_n^\ml-\overline{Y}_n^\ml) \weakarrow B_\br^\mr.$$ The stated claim now easily follows.
\end{proof}

\begin{proof}[\textbf{Proof of Proposition \ref{prop:Wlim}}]

Using Lemma \ref{lem:WnewLaw}, we have
\begin{align*}
    \frac{1}{m\sqrt{n}} W_{n,m} \overset{d}{=}&\bigg(\frac{1}{m\sqrt{n}} \sum_{j=0}^n (Y_n^\ml(j)-\frac{m}{n}j) + \frac{\sqrt{n}}{m} (Y_n^\ml(\tau_n) - \frac{m}{n}\tau_n)\\
        &+\frac{1}{m\sqrt{n}} \sum_{j=0}^n (Y_n^\mr(j)-\frac{n}{m}j) + \frac{1}{\sqrt{n}} \left(Y_n^\mr\circ Y^{\ml} (\tau_n) - \frac{n}{m}Y_n^\ml(\tau_n) \right) + o(1)\Bigg) \big|A_{n,m}.
\end{align*}
By Lemma \ref{lem:helps1}, 
\begin{align*}
    \frac{1}{m\sqrt{n}}W_{n,m} \weakarrow &\int_0^1 \left(\sqrt{\alpha} B_\br^\ml(t) + B^\mr_\br(t)\right)\,dt - \sqrt{\alpha} B_\br^\ml(\tau) - B^\mr_\br(\tau)\\
    &= \int_0^1\left( \cS(t)-\cS(\tau)\right)\,dt \overset{d}{=} \sqrt{1+\alpha} \int_0^1 B_{\operatorname{ex}}(t)\,dt.
\end{align*}
The convergence of moments follows easily from the uniform bound in $L^p$ in Lemma \ref{lem:helps1} and, for example, Theorem 3.5 and equation (3.18) in \cite{Billingsley.99}.
\end{proof}

\subsection{Proof of Theorem \ref{thm:main}}

The proof of Theorem \ref{thm:main} is now straight-forward. Propositions \ref{prop:Radon-Nikodym} and \ref{prop:Wlim} imply
\begin{equation*}
    \frac{\#\bbG_{n,m}(k)}{\#\bbG_{n,m}(0)} = \E\left[\binom{W_{n,m}}{k}\right] \sim \frac{1}{k!}(1+\alpha)^{k/2} \E\left[\left(\int_0^1 B_{\operatorname{ex}}(t)\,dt\right)^k\right].
\end{equation*}
This is the desired result.

\bibliography{ref, references}

\providecommand{\bysame}{\leavevmode\hbox to3em{\hrulefill}\thinspace}
\providecommand{\MR}{\relax\ifhmode\unskip\space\fi MR }
\providecommand{\MRhref}[2]{%
  \href{http://www.ams.org/mathscinet-getitem?mr=#1}{#2}
}
\providecommand{\href}[2]{#2}
\begin{thebibliography}{10}

\bibitem{ABBG.12}
L.~Addario-Berry, N.~Broutin, and C.~Goldschmidt, \emph{The continuum limit of critical random graphs}, Probab. Theory Related Fields \textbf{152} (2012), no.~3-4, 367--406. \MR{2892951}

\bibitem{AUB.20}
Osvaldo {Angtuncio} and Ger{\'o}nimo {Uribe Bravo}, \emph{{On the profile of trees with a given degree sequence}}, arXiv e-prints (2020), arXiv:2008.12242.

\bibitem{Bertoin.01}
Jean Bertoin, \emph{Eternal additive coalescents and certain bridges with exchangeable increments}, Annals of probability (2001), 344--360.

\bibitem{Billingsley.99}
Patrick Billingsley, \emph{Convergence of probability measures}, second ed., Wiley Series in Probability and Statistics: Probability and Statistics, John Wiley \& Sons, Inc., New York, 1999, A Wiley-Interscience Publication. \MR{1700749}

\bibitem{Clancy.24_Bipartite}
David {Clancy, Jr.}, \emph{{Near-critical bipartite configuration models and their associated intersection graphs}}, arXiv e-prints (2024), arXiv:2410.11975.

\bibitem{DEKM.23}
Tuan~Anh Do, Joshua Erde, Mihyun Kang, and Michael Missethan, \emph{Component behaviour and excess of random bipartite graphs near the critical point}, Electron. J. Combin. \textbf{30} (2023), no.~3, Paper No. 3.7, 53. \MR{4614541}

\bibitem{Federico.19}
Lorenzo Federico, \emph{Critical scaling limits of the random intersection graph}, arXiv preprint arXiv:1910.13227 (2019).

\bibitem{HSY.22}
Taro Hasui, Tomoyuki Shirai, and Satoshi Yabuoku, \emph{Enumeration of connected bipartite graphs with given betti number}, arXiv preprint arXiv:2208.03996 (2022).

\bibitem{JS.13}
Jean Jacod and Albert Shiryaev, \emph{Limit theorems for stochastic processes}, vol. 288, Springer Science \& Business Media, 2013.

\bibitem{Janson.07}
Svante Janson, \emph{Brownian excursion area, {W}right's constants in graph enumeration, and other {B}rownian areas}, Probab. Surv. \textbf{4} (2007), 80--145. \MR{2318402}

\bibitem{Kersting.11}
G{\"o}tz {Kersting}, \emph{{On the Height Profile of a Conditioned Galton-Watson Tree}}, arXiv e-prints (2011), arXiv:1101.3656.

\bibitem{LeGall.05}
Jean-Fran\c{c}ois Le~Gall, \emph{Random trees and applications}, Probab. Surv. \textbf{2} (2005), 245--311. \MR{2203728}

\bibitem{Marckert.08}
Jean-Fran\c{c}ois Marckert, \emph{One more approach to the convergence of the empirical process to the {B}rownian bridge}, Electron. J. Stat. \textbf{2} (2008), 118--126. \MR{2386089}

\bibitem{Pal.12}
Soumik Pal, \emph{Brownian approximation to counting graphs}, SIAM J. Discrete Math. \textbf{26} (2012), no.~3, 1181--1188. \MR{3022133}

\bibitem{Pitman.06}
J.~Pitman, \emph{Combinatorial stochastic processes}, Lecture Notes in Mathematics, vol. 1875, Springer-Verlag, Berlin, 2006, Lectures from the 32nd Summer School on Probability Theory held in Saint-Flour, July 7--24, 2002, With a foreword by Jean Picard. \MR{2245368}

\bibitem{Scoins.62}
Hubert~Ian Scoins, \emph{The number of trees with nodes of alternate parity}, Mathematical Proceedings of the Cambridge Philosophical Society, vol.~58, Cambridge University Press, 1962, pp.~12--16.

\bibitem{Spencer.97}
Joel Spencer, \emph{Enumerating graphs and brownian motion}, Communications on Pure and Applied Mathematics: A Journal Issued by the Courant Institute of Mathematical Sciences \textbf{50} (1997), no.~3, 291--294.

\bibitem{Vervaat.72}
Wim Vervaat, \emph{Functional central limit theorems for processes with positive drift and their inverses}, Z. Wahrscheinlichkeitstheorie und Verw. Gebiete \textbf{23} (1972), 245--253. \MR{321164}

\bibitem{Vervaat.79}
\bysame, \emph{A relation between {B}rownian bridge and {B}rownian excursion}, Ann. Probab. \textbf{7} (1979), no.~1, 143--149. \MR{515820}

\bibitem{Wang.23}
Minmin Wang, \emph{Large random intersection graphs inside the critical window and triangle counts}, arXiv preprint arXiv:2309.13694 (2023).

\bibitem{Whitt.80}
Ward Whitt, \emph{Some useful functions for functional limit theorems}, Mathematics of operations research \textbf{5} (1980), no.~1, 67--85.

\bibitem{Whitt.02}
\bysame, \emph{Stochastic-process limits}, Springer Series in Operations Research, Springer-Verlag, New York, 2002, An introduction to stochastic-process limits and their application to queues. \MR{1876437}

\bibitem{Wright.77}
E.~M. Wright, \emph{The number of connected sparsely edged graphs}, J. Graph Theory \textbf{1} (1977), no.~4, 317--330. \MR{463026}

\end{thebibliography}
\bibliographystyle{amsplain}

\end{document}